\documentclass[11pt,a4paper]{amsart}

\usepackage{fullpage}
\usepackage{amsmath}
\usepackage{amsfonts}
\usepackage{amssymb}
\usepackage{mathrsfs}
\usepackage{mathtools}
\usepackage{bbm}
\usepackage{tikz-cd}
\usepackage{url}
\usepackage[english]{babel}
\usepackage[utf8]{inputenc}
\usepackage{csquotes}
\usepackage{enumitem}
\usepackage{multicol}
\usepackage{calc}

\usepackage{array}

\setlist[itemize,1]{label={--\,}}
\setlist[enumerate,1]{label=(\roman*)}

\usepackage[backend=biber, maxbibnames=50, style=alphabetic, giveninits=true]{biblatex}
\renewbibmacro{in:}{}

\usepackage{hyperref}
\usepackage{thmtools}
\usepackage{cleveref}

\DeclareRobustCommand{\gobblefour}[4]{}

\makeatletter
\renewcommand{\tocsection}[3]{%
  \indentlabel{\@ifnotempty{#2}{\ignorespaces#1 \makebox[\widthof{00.}][l]{#2.}\quad}}#3}
\renewcommand{\tocsubsection}[3]{%
  \indentlabel{\@ifnotempty{#2}{\ignorespaces#1 \makebox[\widthof{00.0.}][l]{#2.}\quad}}#3}
\makeatother

\setlist[itemize]{leftmargin=*}
\setlist[enumerate]{leftmargin=*}

\newtheorem{thm}{Theorem}[section]
\newtheorem{defin}[thm]{Definition}
\newtheorem{conj}[thm]{Conjecture}
\newtheorem{prop}[thm]{Proposition}
\newtheorem{cor}[thm]{Corollary}

\newtheorem{assumpt}[thm]{Assumption}

\theoremstyle{definition}

\newtheorem{rem}[thm]{Remark}

\newtheorem{innercustomconj}{Conjecture}
\newenvironment{customconj}[1]
  {\renewcommand\theinnercustomconj{#1}\innercustomconj}
  {\endinnercustomconj}
  
\newcommand{\Q}{\mathbb Q}

\newcommand{\Z}{\mathbb Z}

\newcommand{\cyc}{\mathrm{cyc}}

\newcommand{\Fil}{\mathrm{Fil}}
\newcommand{\Gal}{\mathrm{Gal}}
\newcommand{\GL}{\mathrm{GL}}

\newcommand{\SL}{\mathrm{SL}}

\DeclareMathOperator{\Spm}{Spm}

\newcommand{\into}{\hookrightarrow}
\newcommand{\onto}{\twoheadrightarrow}
\newcommand{\ovl}{\overline}
\newcommand{\wtl}{\widetilde}

\newcommand{\ccirc}{\kern0.5ex\vcenter{\hbox{$\scriptstyle\circ$}}\kern0.5ex}

\newcommand{\bP}{\mathbb{P}}

\newcommand{\calH}{\mathcal{H}}
\newcommand{\cL}{\mathcal{L}}
\newcommand{\cM}{\mathcal{M}}
\newcommand{\cS}{\mathcal{S}}
\newcommand{\cO}{\mathcal{O}}

\newcommand{\Qp}{\overline{\Q}_p}

\newcommand{\vareps}{\varepsilon}

\newcommand{\Tr}{{\mathrm{Tr}}}

\title{$\cL$-Invariants and deep congruences between newforms}
\author{Andrea Conti and Peter Mathias Gr\"af}

\begin{document}

\begin{abstract}
We study congruences modulo powers of a prime $p$ between pairs of $p$-new modular Hecke eigenforms of level $\Gamma_0(p)$ and same weight $k$. Based on explicit computations, we conjecture that every such eigenform $f$ admits a twin to which it is congruent modulo a surprisingly high power of $p$, whose exponent is close to the opposite of the valuation of the $\cL$-invariant of $f$, and whose Atkin--Lehner sign is opposite to that of $f$. This is a new phenomenon that is not explained by the known results on the $p$-adic variation of eigenforms. Inspired by the global picture, we formulate a local conjecture describing congruences between semistable representations of fixed weight, varying $\cL$-invariant, and opposite Atkin--Lehner signs. We give some theoretical evidence towards our conjectures.
\end{abstract}

\maketitle

\section*{Introduction}

Let $f$ be a modular Hecke eigenform of weight $k> 2$ and level $\Gamma_0(p)$ for some prime $p$. Assume that $f$ is $p$-new, which implies in particular that its $U_p$-eigenvalue is $\pm p^{\frac{k-2}{2}}$. It is understood by work of Hida, Coleman-Mazur and Buzzard that if $k$ is replaced by a weight $k'$ that is $p$-adically very close to $k$, one can find an eigenform $f'$ whose Hecke eigensystem is highly congruent to that of $f$: this can be deduced by the much stronger statement that such eigenforms live in $p$-adic families over the rigid analytic eigencurve. On the other hand, if the weight $k$ is fixed, there is no known reason to expect any congruences between the Hecke eigensystems of two $p$-new eigenforms modulo a high power of $p$. In this work, we gather computational evidence towards the fact that such ``deep congruences'' actually occur, in a systematic way and modulo powers of $p$ that increase with the weight: namely, most $p$-new eigenforms admit another eigenform to which they are deeply congruent, and the depth can be controlled in terms of the invariants of the form. The fact that such congruences are ubiquitous is striking evidence for the existence of a hidden structure justifying them, though we do not know at the moment what this could be. Given the long line of research stemming from Mazur's seminal paper \cite{mazureis} around modulo $p$ congruences between Eisenstein series and cuspidal newforms, one might expect the theory behind higher congruences to be quite deep.

We present our work in more detail. Recall that, for a $p$-new eigenform $f$, all the information about the $U_p$-eigenvalue $a_p=\pm p^{\frac{k-2}{2}}$ is contained in the Atkin--Lehner sign $\vareps$. In particular, it is unreasonable to expect $a_p$ itself to contain much information about $p$-adic congruences among Hecke eigensystems, in contrast to the case of the $p$-adic variation of Hecke eigensystems along the eigencurve, which is essentially governed by the variation of $k$ and $a_p$. A finer invariant in the $p$-new setting is the \emph{$\cL$-invariant} $\cL(f)$, as defined in various equivalent ways by Mazur--Tate--Teitelbaum \cite{mtt}, Greenberg--Stevens \cite{greensteve}, and Mazur via Fontaine's theory \cite{mazurinv}. In particular, the knowledge of $k,\vareps$ and $\cL(f)$ amounts to that of the local at $p$, $p$-adic Galois representation attached to $f$: it is the semistable representation  $V_{k,\vareps,\cL(f)}$ attached to these data. While it is expected that keeping $k$ and $\vareps$ fixed and slightly perturbing $\cL(f)$ would provide us with highly congruent local Galois representations, the striking result of our computation is that the valuation of $\cL(f)$ alone determines how deeply the Hecke eigensystem of $f$ (a global object) is congruent to that of a ``twin'' form with \emph{opposite} Atkin--Lehner sign. We summarize our findings in the following conjecture. We say that an $\cL$-invariant $\cL$ is \emph{admissible} if the valuation $v_p(\cL)$ is smaller than the negative constant
\begin{equation*} -C_{p,k}=-\lfloor\log_p\left(\frac{k-2}{p-1}\right)\rfloor-5 \end{equation*}
coming from the work of Jiawei An \cite{jiawei} (see \Cref{sec:jiawei}). We say that a $p$-newform $f$ is admissible if $\cL(f)$ is. According to Bergdall and Pollack's conjecture \cite[Section 8.2]{bplinv}, proved by An \cite[Theorem 1.9]{jiawei} conditionally on the ghost conjecture, one expects the valuations of the $\cL$-invariants to be equidistributed over the interval $[-k(p-1)/(2(p+1)),0]$ as $k$ goes to $\infty$. In particular, one expects only $O(\log_p(k))$ eigenforms to be non-admissible in weight $k$.

We make the recurring assumption that $p$ splits completely in the Hecke field $K_f$. This is mostly for computational purposes, and simplifies the discussion of congruences, but we expect congruences to appear independently of the behaviour of $p$ in $K_f$.

\begin{customconj}{1}\label{globintro}
Let $f$ be an admissible newform of weight $k> 2$ and level $\Gamma_0(p)$, such that $p$ splits completely in $K_f$. 
Then there exists a unique newform $g$ of weight $k$ and level $\Gamma_0(p)$ such that 
\begin{enumerate}
\item $\vareps(g)=-\vareps(f)$;
\item $v_p(\cL(f)+\cL(g))\ge -C_{p,k}$ (in particular, $v_p(\cL(f))=v_p(\cL(g))$); 
\item the Hecke eigensystems of $f$ and $g$ are congruent modulo $p^{-v_p(\cL(f))+1}$.
\end{enumerate}
\end{customconj}

We gathered a reasonably large amount of data in support of \Cref{globintro}: we computed up to weight 200 for $p=3$, 150 for $p=5$, and 100 for $p=7$ and 11. We have some data in tame level larger than 1 that also goes in a similar direction, but since it is not as extensive as that in tame level 1, we did not include it in the formulation of our conjecture. In particular, it is unclear whether one could expect uniqueness of the eigenform $g$ for higher tame level, since repeated $\cL$-invariant valuations for the same Atkin--Lehner sign do not appear in our data.

In view of the aforementioned equidistribution conjecture of Bergdall--Pollack, and result of An, (iii) of \Cref{globintro} predicts that we can produce congruences modulo powers of $p$ that increase linearly in $k$, and (ii) implies that there is typically quite a lot of cancellation happening when taking the sum of the two $\cL$-invariants. The actual cancellation we detect in the computation is actually always strictly larger than $-C_{p,k}$, so the constant does not seem to be optimal; we have however much less data on cancellations of sums of $\cL$-invariants than on congruences of newforms. For instance, for $p=5$ and $k=32$, we have two forms with $\cL$-invariant valuation $-11$, that are congruent modulo $p^{12}$, and for which the sum of the $\cL$-invariants has valuation $-2$. We present in \Cref{sec:data} some samples of the data we computed, and refer to \cite{deepcong} for more of it. 

We gather in \Cref{sec:theo} some theoretical evidence towards Conjecture \ref{globintro}. Among such evidence, one can extract from \cite{jiawei} the fact that $p$-new eigenforms appear in pairs $f,g$ satisfying $v_p(\cL(f))=v_p(\cL(g))$ (see \Cref{thm:an}). The existence of such pairs is deduced from reading the valuations off of Bergdall and Pollack's ghost series for the $U_p$-operator, via the so-called phenomenon of ``ghost duality''. 

We deduce some modulo $p$ congruence in small weight from the work of Breuil--Mézard (see \Cref{prop:bremez}), which has the interesting feature of making the Atkin--Lehner sign appear in the description of the modulo $p$ local Galois representation, giving some support to property $(i)$ of Conjecture \ref{globintro}. The fact that the forms in the congruent pairs have opposite signs would also be explained in small weight by a conjecture of Calegari--Stein (\Cref{conj:CS}) and is compatible with a result of Rozensztaijn (\Cref{prop:rozen}) and with known equidistribution results for the Atkin--Lehner sign (see for example \cite[Corollary 2.14]{HLS} for square-free level).

Since the datum of $k,\vareps,\cL$ determines a local Galois representation $V_{k,\vareps,\cL}$, one can try to get a more accessible statement than Conjecture \ref{globintro} by replacing global congruences of Hecke eigensystems with local congruences of Galois representations. In \Cref{sec:repcong}, we recall and compare different kinds of congruence for group representations, the simplest one being trace-congruence. We encode such an expectation in the following conjecture. 
Let $\cL,\cL^\prime\in\Q_p$ and 
let $h\ge 1$ be an integer.

\begin{customconj}{2}\label{locintro}
If $\cL,\cL'$ are admissible and $v_p(\cL+\cL^\prime)\ge -C_{p,k}$, then $V_{k,\cL,1}$ and $V_{k,\cL^\prime,-1}$ are congruent modulo $p^{-v_p(\cL)+1}$.
\end{customconj}

We do not know at the moment how to prove any case of \Cref{locintro}. We remark in \Cref{sec:locglob} that, for $p<11$, the conjecture would imply Conjecture \ref{globintro}, since for such a small prime a $p$-adic representation of $\Gal(\ovl\Q/\Q)$, unramified away from $p\infty$, is uniquely determined by its restriction to a decomposition group at $p$ by a result of Chenevier \cite{chenquelques}.

We include a discussion of congruences between semistable representations with the same Atkin--Lehner sign. Even if this does not give any evidence towards \Cref{locintro}, we think it might serve as an interesting comparison. Semistable representations of a fixed weight $k$ and Atkin--Lehner sign $\vareps$ are parameterized by the $\cL$-invariant, and form in a suitable sense a family over $\bP^1(\Qp)$. In particular, the reduction of the representations in the family modulo a power of $p$ is locally constant, and can be deduced from knowledge of the domains on which the mod $p$ reduction is constant via results of the first-named author with Torti \cite{contor}. When $2<k<p$, such domains admit a very simple description thanks to the results of Breuil--Mézard \cite{bremez} and Guerberoff--Park \cite{gueparht}, and we can determine explicit radii of mod $p^n$ congruence. Unfortunately, there is no way to go between the two projective lines while controlling the mod $p^n$ reduction, so that it is not clear whether a same-sign congruence gives any information towards \Cref{locintro}.

The last part of the paper is devoted to a short discussion of our computations in SageMath and Magma, and presenting samples of the extensive data we have in support of Conjecture \ref{globintro}.


\smallskip

\noindent\textbf{Conventions.} For every field $K$, we denote by $\ovl K$ an algebraic closure of $K$ and by $G_K$ the absolute Galois group $\Gal(\ovl K/K)$.

\noindent For every prime $p$, we fix a field embedding $\ovl\Q\into\Qp$, identifying $G_{\Q_p}$ with a decomposition subgroup of $G_\Q$ at $p$. We choose a $p$-adic valuation on $\Qp$ such that $v_p(p)=1$.
We denote by $\chi_\cyc$ the $p$-adic cyclotomic character of $G_\Q$, as well as its restrictions to subgroups of $G_\Q$.

\smallskip

\noindent\textbf{Acknowledgments.} Both authors acknowledge support from the Deutsche Forschungsgemeinschaft - Project numbers 444845124 (first author) and 491064713, 546550122 (second author). \\ 
We wish to thank Jiawei An, John Bergdall, Gebhard B\"ockle, Alexandre Maksoud, Anna Medvedovsky, Rob Pollack and Gabor Wiese for various helpful exchanges and for their comments on a first draft of the paper. Our computations were carried out on the computing servers of Boston University, Bonn University and Heidelberg University. We thank these institutions for the use of their machines.

\tableofcontents

\section{Conjectural deep congruences between newforms}

The goal of this section is to state the ``global'' Conjecture \ref{globconj} about congruences between Hecke eigensystems of newforms, or equivalently between the traces of global $p$-adic Galois representations attached to newforms, and the ``local'' Conjecture \ref{locconj} about congruences of $p$-adic representations of $G_{\Q_p}$, semistable at $p$. We explain how the local conjecture implies the global one for small values of $p$, thanks to a result of Chenevier.

Throughout this section, we fix an integer weight $k> 2$ and a prime $p$. 

\subsection{Congruences of Hecke eigensystems and Galois representations}\label{sec:repcong}

We start with some definitions and facts about congruences between either representations of profinite groups or their traces. 
Let $\calH$ be the abstract Hecke algebra over $\Z$, spherical outside $p$ and of level $\Gamma_0(p)$ at $p$: it is generated by the operators $T_\ell$ at the primes $\ell\neq p$, and the operator $U_p$. Let $\calH^p$ be the spherical factor of $\calH$.

\begin{defin}
	A ($\Z_p$-valued) \emph{spherical Hecke eigensystem} is a ring homomorphism $\calH^p\to\Z_p$.
	For $n\in\Z$, $n\ge 1$, we say that two spherical Hecke eigensystems $\alpha,\beta$ are congruent modulo $p^n$ if $v_p(\alpha(T)-\beta(T))\ge n$ for every $T\in\calH^p$.
\end{defin}

\begin{rem}
Two Hecke eigensystems $\alpha,\beta$ are congruent modulo $p^n$ if and only if $v_p(\alpha(T_\ell)-\beta(T_\ell))\ge n$ for every $\ell\ne p$. 
\end{rem}

Let $G$ be a profinite group, and let $\rho,\rho^\prime\colon G\to\GL_2(\Q_p)$ be two continuous representations. Let $n\in\Z, n\ge 1$. 

\begin{defin}
We say that $\rho,\rho^\prime$ are \emph{trace-congruent modulo $p^n$} if, for every $g\in G$, $v_p(\Tr(\rho(g))-\Tr(\rho^\prime(g)))\ge n$.
\end{defin}

A standard argument shows that $\Q_p^2$, equipped with any continuous action of $G$, admits a $G$-stable $\Z_p$-lattice. The modulo $p$ reduction of such a lattice, considered as a $\Z/p$-linear representation of $G$, is independent of the lattice only up to semisimplification. Along the lines of \cite[Definition 1.1]{bellconv}, we give the following.

\begin{defin}
Two continuous representations $\rho,\rho^\prime\colon G\to\GL_2(\Q_p)$ are \emph{physically congruent modulo $p^n$} if there exist lattices for $\rho$ and $\rho^\prime$ whose reductions modulo $p^n$ give isomorphic $\Z/p^n$-linear representations of $G$. 

Two continuous representations $\rho,\rho^\prime\colon G\to\GL_2(\Z_p)$ are \emph{physically congruent modulo $p^n$} if their reductions modulo $p^n$ give isomorphic $\Z/p^n$-linear representations of $G$. 
\end{defin}

In order to make some later statements simpler, we do not allow further conjugation if $\rho,\rho^\prime$ are taken with coefficients in $\Z_p$ to start with. 

By a result of Carayol \cite[Théorème 1]{cartrace}, if $\ovl\rho$ and $\ovl\rho^\prime$ are absolutely irreducible, then $\rho$ and $\rho^\prime$ are physically congruent modulo $p^n$ if and only if they are trace-congruent modulo $p^n$. 

Throughout this article, we denote by $\cS_k(\Gamma_0(p))$ the space of cusp forms of weight $k$ and level $\Gamma_0(p)$ and by $\cS_k(\Gamma_0(p))^\mathrm{new}\subset \cS_k(\Gamma_0(p))$ the subspace of newforms. In the sequel, by a \emph{newform} we mean a normalized Hecke eigenform in $\cS_k(\Gamma_0(p))^\mathrm{new}$. 
For every such newform $f$, we denote by $K_f$ its coefficient field, by $\vareps(f)\in\{\pm 1\}$ its Atkin--Lehner sign, and by $\cL(f)\in K_{f,\mathfrak{p}}$ its $p$-adic $\cL$-invariant (see Section \ref{sec:Linv} for more details). Here, $K_{f,\mathfrak{p}}$ denotes the completion of $K_f$ and $\mathfrak{p}\mid p$. We can read $\vareps(f)$ off the $U_p$-eigenvalue $a_p$ of $f$, since $a_p=-\vareps(f)p^{\frac{k-2}{2}}$. 

Assume that $p$ splits completely in $K_f$. We write $\alpha(f)\colon\calH^p\to\Z_p$ for the system of spherical Hecke eigenvalues of $f$, and $\rho_{f,p}\colon G_\Q\to\GL_2(\Q_p)$ for the $p$-adic Galois representation attached to $f$. Since the traces of $\rho_{f,p}$ at unramified places are given by the spherical Hecke eigenvalues of $f$, the following remark is immediate.

\begin{rem}\label{hecketogal}
Let $n\in\Z, n\ge 1$. Let $f,g$ be two newforms of weight $k$ and level $\Gamma_0(p)$, such that $p$ splits completely in the respective Hecke fields. Then $\alpha(f)$ is congruent to $\alpha(g)$ modulo $p^n$ if and only if $\rho_{f,p}$ and $\rho_{g,p}$ are trace-congruent modulo $p^n$.
\end{rem}

\subsection{Reminders on $\cL$-invariants}
\label{sec:Linv}
We start by recalling how semistable, non-crystalline $p$-adic representations of $G_{\Q_p}$ are classified by their weight, $\cL$-invariant and Atkin--Lehner sign, as summarized for instance in \cite[Exemple 3.1.2.2]{bremez}. Throughout this section, let $k> 2$ be an integer. 

We denote by $\bP^1$ the rigid analytic projective line over $\Q_p$, and by $\cL$ a variable on it. 
Let $\varpi$ be a square root of $p$ in $\Qp$. Let $\vareps\in\{\pm 1\}$, and $\cL\in\bP^1(\Qp)=\Qp\cup\{\infty\}$.
We denote by $V_{k,\mathcal{L},\vareps}$ the dual of the semistable representation of $G_{\Q_p}$ whose associated $(\varphi,N)$-module is a 2-dimensional $\Qp$-vector space $D_{k,\cL,\vareps}$ equipped with the structures defined, in a basis $(e_1,e_2)$, by
\[ \varphi=\begin{pmatrix}\vareps\varpi^{k-2} & 0 \\ 0 & \vareps\varpi^{k-2}\end{pmatrix},\quad N=\begin{pmatrix}0 & 0 \\ 1 & 0\end{pmatrix}, \quad \Fil^iD_{k,\cL,\vareps}=\begin{cases}D_{k,\cL,\vareps}\text{ if }i\le 0 \\ \Qp(e_1+\cL e_2)\text{ if }1\le i\le k-1 \\ 0\text{ if }i\ge k\end{cases} \]
if $\cL\in\Qp$, and by
\[ \varphi=\begin{pmatrix}\varpi^{k-2} & 0 \\ 0 & \varpi^{k-2}\end{pmatrix},\quad N=0, \quad \Fil^iD_{k,\cL,\vareps}=\begin{cases}D_{k,\cL,a}\text{ if }i\le 0 \\ \Qp(e_1+e_2)\text{ if }1\le i\le k-1 \\ 0\text{ if }i\ge k\end{cases} \]
if $\cL=\infty$. 
The above description for $V_{k,\infty}$ can be obtained by rewriting $V_{k,\cL}$, $\cL\in\Qp$, in the basis $(e_1,\cL e_2)$, and letting $\cL\to\infty$. 
For $\cL\ne\infty$, the representations $V_{k,\cL,\vareps}$ are semistable and non-crystalline of Hodge--Tate weights $(0,k-1)$. By a standard calculation in Fontaine's theory, every semistable, non crystalline representation of $G_{\Q_p}$ of Hodge--Tate weights $(0,k-1)$ is, up to twist with a semistable character of $G_{\Q_p}$, isomorphic to $V_{k,\cL,\vareps}$ for a unique $\cL\in\Qp$. On the other hand, $V_{k,\infty,\vareps}$ is crystalline.


\subsection{Doubling of $\cL$-invariants}\label{sec:jiawei}

For a prime $p$ and a positive integer $k$, set
\begin{equation}\label{Cpk} C_{p,k}=\lfloor\log_p\left(\frac{k-2}{p-1}\right)\rfloor+5. \end{equation}
Note that we are assuming $k>2$, so the logarithm is well defined. We introduce the following terminology in order to simplify the exposition throughout the paper.

\begin{defin}
We call an $\cL$-invariant $\cL$ \emph{admissible} if $v_p(\cL) < -C_{p,k}$. We call a newform $f\in\cS_k(\Gamma_0(p))^\mathrm{new}$ admissible if $\cL(f)$ is admissible.
\end{defin}

According to Bergdall and Pollack's conjecture \cite[Section 8.2]{bplinv}, proved by An in \cite[Theorem 1.9]{jiawei} conditionally on the ghost conjecture, one expects the valuations of the $\cL$-invariants to be equidistributed over the interval $[-k(p-1)/(2(p+1)),0]$ as $k$ goes to $\infty$. In particular, one expects only $O(\log_p(k))$ eigenforms to be non-admissible in weight $k$.

We record another result of An. We came to learn of the following formulation from a talk of Rob Pollack, and we thank him for explaining to us how to extract it from \cite{jiawei}. 
Even though it was proved after we had already formulated some of our conjectures based on computational evidence, and provided some theoretical evidence for them, we prefer to present it first since this will help streamline the exposition of our work. 

\begin{thm}[{\cite{jiawei}}]\label{thm:an}
Let $p\ge 11$, and let $f$ be an admissible newform of level $\Gamma_0(p)$ and weight $k$. Assume that $\ovl\rho_f\vert_{I_{\Q_p}}$ is locally reducible and strongly generic in the sense of \cite[Definition 1.5]{jiawei}, i.e. isomorphic to a twist of
    \[ \begin{pmatrix} \ovl\chi_\cyc^{a+1} & \ast \\ 0 & 1\end{pmatrix}, \]
    with $2\le a \le p-5$.
Then there exists a newform $g\ne f$ of level $\Gamma_0(p)$ and weight $k$ such that:
\begin{itemize}
\item $\ovl\rho_f\cong\ovl\rho_g$;
\item $v_p(\cL(f))=v_p(\cL(g))$.
\end{itemize}
\end{thm}

We briefly explain how to extract the result and the constant $C_{p,k}$ from \cite{jiawei}. In the paper, An compares four different invariants: the valuation of the $\cL$-invariants, the slopes of the matrix $A_1$, the $k$-thresholds, and the slopes of the polygon $\underline{\Delta}^+_k$. These invariants all coincide after removing those that are non-admissible (larger than a constant $-C_{p,k}$: \cite[Corollary 3.25(1)]{jiawei} compares the slopes of $\underline{\Delta}^+_k$ (denoted there by $s_j$) with the $k$-thresholds, \cite[Theorem 1.17]{jiawei} compares the valuations of the $\cL$-invariants with the slopes of $A_1$, and \cite[Theorem 1.18]{jiawei} compares the slopes of $A_1$ with the $k$-thresholds.  It is shown in \cite[Corollary 3.25(1)]{jiawei} that the $k$-thresholds (hence the valuation of the $\cL$-invariants) come in pairs as given in \Cref{thm:an}: this is a consequence of the ``ghost duality'' that gives the corresponding doubling of the slopes of  $\underline{\Delta}^+_k$.

The forms that do not fit in the pattern of ``doubled'' $\cL$-invariants are (at most) those for which the $\cL$-invariants are larger than the constant $-R-1$ from \cite[Theorem 5.32(2)]{jiawei}. Now by \cite[Remark 4.22(3)]{jiawei}, $R<M(k)+1$, and by \cite[Lemma 3.17]{jiawei}, $M(k)<\lfloor\log{\left(\frac{k-k_0}{p-1}\right)}\rfloor+3$, where $k_0$ is the unique integer in $\{2,\ldots,p\}$ congruent to $k$ modulo $p-1$. We deduce that
\[ -R-1>-M(k)-2>-\lfloor\log{\left(\frac{k-k_0}{p-1}\right)}\rfloor-5=-\lfloor\log{\left(\frac{k-2}{p-1}\right)}\rfloor-5. \]


\subsection{A global conjecture}

We make the following assumption on any newform $f\in\cS_k(\Gamma_0(p))^\mathrm{new}$ appearing in the rest of the paper.

\begin{assumpt}
\label{main-ass}
The prime $p$ splits completely in the Hecke field $K_f$ of $f$.
\end{assumpt}

We refer the reader to Remark \ref{rem:main-ass} for a discussion of Assumption \ref{main-ass}. 
%
%
Based on the data we computed, we formulate the following conjecture. 



\begin{conj}\label{globconj}
Let $f$ be an admissible newform of weight $k> 2$ and level $\Gamma_0(p)$, such that $p$ splits completely in $K_f$. 
Then there exists a unique 
newform $g$ of weight $k$ and level $\Gamma_0(p)$ such that 
\begin{enumerate}
\item $\vareps(g)=-\vareps(f)$;
\item $v_p(\cL(f)+\cL(g))\ge -C_{p,k}$ (in particular, $v_p(\cL(f))=v_p(\cL(g))$); 
\item\label{ppower} the Hecke eigensystems of $f$ and $g$ are congruent modulo $p^{-v_p(\cL(f))+1}$.
\end{enumerate}
\end{conj}

\noindent Even if we do not make this into a formal definition, it will sometimes be convenient to say that two forms $f,g$ satisfying point \ref{ppower} of Conjecture \ref{globconj} are \emph{deeply congruent}.

\begin{rem}\mbox{ }\label{rem:glob}
\begin{enumerate}
\item Since $f$ is admissible, condition $(ii)$ implies that $v_p(\cL(f)+\cL(g))>v_p(\cL(f))$, hence $v_p(\cL(f))=v_p(\cL(g))$ and some cancellation happens when computing the sum $v_p(\cL(f)+\cL(g))$. By our definition of admissibility, $C_{p,k}$ is the largest constant for which this is true, which partly justifies its choice. In the examples we compute, however, $v_p(\cL(f)+\cL(g))$ is always strictly larger than $-C_{p,k}$ (see \Cref{sec:cancellation}). As remarked earlier, Bergdall and Pollack conjecture that the valuations of $\cL$-invariants are equidistributed over the interval $[-k(p-1)/(2(p+1)),0]$ as $k$ goes to $\infty$ \cite[Section 8.2]{bplinv}. This has been proved in \cite[Theorem 1.9]{jiawei} conditionally on the ghost conjecture. In particular, (iii) of \Cref{globconj} predicts that we can produce congruences modulo powers of $p$ that increase linearly in $k$, and condition (ii) implies that there is typically quite a lot of cancellation happening when taking the sum of the two $\cL$-invariants. 
\item Via Remark \ref{hecketogal}, we can replace condition $(iii)$ in Conjecture \ref{globconj} with the following: $\rho_{f,p}$ and $\rho_{g,p}$ are trace-congruent modulo $p^{v_p(\cL)}$. This remark inspires us to investigate the weaker condition of the local Galois representations at $p$ being congruent: the advantage in this is that such representations admit a simple parameterization, that we recall in the next section.
\item\label{glob2} If we remove condition $(iii)$, then a form $g$ satisfying $(i,ii)$ is not unique any more, as one sees for instance from Table \ref{11-18}.
\item As exemplified again by the data in Section \ref{sec:data}, the exponent of $p$ in Conjecture \ref{globconj}\ref{ppower} is not optimal for every $f$, and can sometimes be increased to $-v_p(\cL(f))+i$ with $i=2$ or 3. We do not have any conjecture on what the precise value should be. Ideally, one would hope to describe the precise depth in terms of $v_p(\cL(f)+\cL(g))-v_p(\cL)$, i.e. the amount of cancellation happening when taking the sum of the $\cL$-invariant. However, as seen from the data of \Cref{sec:cancellation}, there is no clear relation between the two quantities.
\item The data we gathered in level $\Gamma_1(N)\cap\Gamma_0(p)$, with $N>1$, seems to suggest a similar behaviour to what we outline in Conjecture \ref{globconj}. However, we do not feel we have sufficient data to make a more general statement: for instance, for $N>1$ we do not know whether one could expect uniqueness of the eigenform $g$, since repeated $\cL$-invariant valuations for the same Atkin--Lehner sign do not appear in our data.
\end{enumerate}
\end{rem}

\subsection{Theoretical evidence}\label{sec:theo}

We gather some theoretical evidence towards Conjecture \ref{globconj}.



\subsubsection{Evidence for same $\cL$-invariant valuation} The first piece of evidence is the result of J. An on the doubling of $\cL$-invariants, that we already presented in \Cref{thm:an}. A newform satisfying the assumptions of Theorem \ref{thm:an} admits at least one ``twin'' with the same valuation of the $\cL$-invariant and the same congruence class mod $p$, but An's technique does not seem to give information on why the Atkin--Lehner signs of the two forms should be different, or why $f$ would admit exactly one such twin form deeply congruent to it. Similarly to what we observed in Remark \ref{rem:glob}\ref{glob2}, the data shows that a form $g$ satisfying the conclusion of Theorem \ref{thm:an} is not necessarily deeply congruent to $f$.

\subsubsection{Evidence for opposite Atkin--Lehner signs}

For small weights, we deduce from the calculations in \cite{bremez} that pairs of deeply congruent forms necessarily have opposite Atkin--Lehner signs.

Let $\cL,\cL'\in\Qp$ 
and $\vareps,\vareps'\in\{\pm 1\}$. The following is obtained by simply observing case $(iii)$ of \cite[Théorème 1.2]{bremez}.

\begin{prop}\label{prop:bremez}
Assume that 
\begin{enumerate}[label=(\roman*)]
\item $k$ is an even integer with $2<k<p-1$, 
\item $v_p(\cL)\ge -\frac{k}{2}+2$, 
\item $v_p(\cL)=v_p(\cL')$.
\end{enumerate}
Then any two of the following three conditions imply the remaining one:
\begin{enumerate}[label=(\arabic*)]
    \item $\ovl V_{k,\cL,\vareps}\cong\ovl V_{k,\cL',\vareps'}$,
    \item $v_p(\cL+\cL')-v_p(\cL)\ge 1$,
    \item $\vareps=-\vareps'$.
\end{enumerate}
\end{prop}

If $V_{k,\cL,\vareps}$ and $V_{k,\cL',\vareps'}$ are attached to two admissible modular forms $f,g$ that are congruent mod $p$ and satisfy $v_p(\cL)\ge -\frac{k}{2}+2$ and $v_p(\cL+\cL')>-C_{p,k}$, then $(i,ii,iii)$ and $(1,2)$ of Proposition \ref{prop:bremez} are satisfied, hence $\vareps=-\vareps'$. Note that by \cite[(10), Section 7.1]{bplinv}, one expects that $-\frac{(p-1)k}{2(p+1)}\le v_p(\cL)\le 0$ for ``most'' eigenforms, hence the bound $v_p(\cL)\ge -\frac{k}{2}+2$ does not shrink this range too much.

We further remark that the above discussion, and our findings, are compatible with the following conjecture of Calegari and Stein.

\begin{conj}[{\cite[Conjecture 3]{calstedisc}}]\label{conj:CS}
Let $k<p-1$ and let $f,g$ be two newforms of level $\Gamma_0(p)$. If $f$ and $g$ are congruent modulo $p$, then $\vareps(f)=-\vareps(g)$.
\end{conj}

For more general weights, our findings are compatible with the following result of Rozensztajn.

\begin{prop}[{\cite[Theorem 6.2.1(2)]{rozenext}}]\label{prop:rozen}
Let $f$ be a newform in $S_k(\Gamma_0(p))$. Assume that either $k> 2p+2$, or $\ovl\rho_{f}\vert_{I_p}$ is not in the case ``très ramifié'' of \cite[Théorème 1.2(i)]{bremez}. Assume further that $\rho_{f}\vert_{G_{\Q(\zeta_p)}}$ is absolutely irreducible. Then there exists a newform $g\in S_k(\Gamma_0(p))$ such that $f$ and $g$ are congruent modulo $p$ and $\vareps(f)=-\vareps(g)$.
\end{prop}

We also remark that in the case when $2<k\le 2p+2$ and $\ovl\rho_{f}\vert_{I_p}$ is in the case ``très ramifié'' of \cite[Théorème 1.2(i)]{bremez}, \cite[Theorem 6.2.1(1)]{rozenext} predicts that there is no congruence modulo $p$ between newforms in $S_k(\Gamma_0(p))$ of opposite Atkin--Lehner signs. One can check using \cite[Théorème 1.2(i)]{bremez} that this case does not appear in our data.

We thank Alexandre Maksoud for pointing out the references \cite{calstedisc,rozenext} to us.

\section{Conjectural deep congruences between semistable representations}

\subsection{A local conjecture}

We have the ingredients we need to formulate a conjecture about congruences between semistable representations of $G_{\Q_p}$ of opposite Atkin--Lehner signs. Let $\cL,\cL^\prime\in\Q_p$ and let $h\ge 1$ be an integer.

\begin{conj}\label{locconj}
If $\cL,\cL'$ are admissible and $v_p(\cL+\cL^\prime)\ge -C_{p,k}$, then $V_{k,\cL,1}$ and $V_{k,\cL^\prime,-1}$ are congruent modulo $p^{-v_p(\cL)+1}$.
\end{conj}

\noindent Since $\cL,\cL'$ are admissible, meaning that $v_p(\cL),v_p(\cL')<-C_{p,k}$, the assumptions of \Cref{locconj} imply that $v_p(\cL)=v_p(\cL^\prime)$ 
and that some cancellation happens when one takes the sum of the $\cL$-invariants, normalized by the Atkin--Lehner sign (i.e. the difference of the derivatives of $a_p$, normalized by $p^{\frac{k-2}{2}}$). 

Representations of the same Atkin--Lehner sign $\vareps$ correspond to points of a ``family'' of representations of $G_{\Q_p}$ (in a sense that we do not make precise here) parameterized by the rigid analytic projective line $\bP^1$ over $\Q_p$, with variable given by the $\cL$-invariant. One expects the modulo $p^n$ reduction of such a family to be constant if one stays in a small enough rigid neighborhood of a point in $\bP^1$, i.e. if one makes a sufficiently small $p$-adic perturbation of the $\cL$-invariant. 
We discuss the situation further in Section \ref{secsamesign}. 
However, two representations of opposite Atkin--Lehner signs will live on two distinct $\bP^1$, and we do not know of a way to relate the corresponding points. 

Twisting with the unramified quadratic character $\eta$ of $G_{\Q_p}$ defines a map from the $\bP^1$ of $G_{\Q_p}$-representations with positive Atkin--Lehner sign to the $\bP^1$ with negative sign. However, the original representation won't in general be congruent to its twist modulo a high power of $p$ (not even modulo $p$, unless the reduction is irreducible or a direct sum $\chi\oplus\eta\chi$ for some character $\chi$), so we do not know of a way to exploit this twist towards \Cref{locconj}.


The justification for Conjecture \ref{locconj} comes instead from our global conjecture. Consider two eigenforms $f,g$ satisfying the assumptions and the conclusions of Conjecture \ref{globconj}, and choose $\cL=\cL(f)$, $\cL^\prime=\cL(g)$, $\vareps=\vareps(f)$. 
Then Conjecture \ref{globconj} predicts a trace-congruence modulo $p^{-v_p(\cL)+1}$ between $\rho_{f,p}$ and $\rho_{g,p}$, hence between their restrictions to $G_{\Q_p}$, which in turn satisfy the assumptions of Conjecture \ref{locconj}. 

\subsection{The local conjecture implies a global congruence for $p\le 7$}\label{sec:locglob}

In this section, let $p\in\{2,3,5,7\}$. We write $I_{\Q_p}$ for the inertia subgroup of $G_{\Q_p}$. Recall that we identify $G_{\Q_p}$, hence $I_{\Q_p}$, with a fixed decomposition subgroup of $G_\Q$ at $p$. We denote by $\Q(\mu_p)(p)$ the maximal pro-$p$ extension of $\Q(\mu_p)$ unramified outside of $p$ and the archimedean places.

\begin{prop}\label{loctoglob}
Let $\rho,\rho^\prime\colon G_{\Q,p\infty}\to\GL_2(\Z_p)$ be two continuous representations, and let $n\ge 1$ be an integer. Then $\rho$ and $\rho^\prime$ are physically congruent modulo $p^h$ if and only if their restrictions to $I_{\Q_p}$ are physically congruent modulo $p^h$.
\end{prop}

\begin{proof}
Let $g\in G_{\Q,p\infty}$. By \cite[Corollary 1.4 and Proposition 1.8(i)]{chenquelques}, both $\rho$ and $\rho^\prime$ factor through $\pi\colon G_\Q\onto\Gal(\Q(\mu_p)(p)/\Q)$, and the composition
\begin{equation} I_{\Q_p}\into G_\Q\to\Gal(\Q(\mu_p)(p)/\Q), \end{equation}
is surjective. For every $g\in\Gal(\Q(\mu_p)(p)/\Q)$, let $\wtl g$ be an element of $I_{\Q_p}$ mapping to $g$. Then, for every $s\in G_\Q$, $\rho(s)=\rho(\wtl{\pi(g)})$ and $\rho^\prime(s)=\rho^\prime(\wtl{\pi(g)})$. 

Now assume that $\rho\vert_{I_{\Q_p}}$ and $\rho^\prime\vert_{I_{\Q_p}}$ are physically congruent modulo $p^h$. Then, for every $s\in G_\Q$, $\rho(\wtl{\pi(g)})$ and $\rho^\prime(\wtl{\pi(g)})$ are congruent modulo $p^h$, which implies that $\rho(s)$ and $\rho(s^\prime)$ are also congruent modulo $p^h$.
\end{proof}

\begin{rem}\mbox{ }
\begin{enumerate}
\item It is essential to the above argument that the element $\wtl{\pi(g)}$ only depends on $g\in G_\Q$, and not on the chosen representation of $G_\Q$.
\item In the statement of Proposition \ref{loctoglob}, we start with $\Z_p$-valued representations: if we started with $\Q_p$-valued representations, we would encounter the problem that physical congruence after restriction to $I_{\Q_p}$ is defined in terms of an $I_{\Q_p}$-stable lattice, which might not be $G_\Q$-stable.
\end{enumerate}
\end{rem}

We deduce the following corollary. Let $f,g$ be two newforms of weight $k$ and level $\Gamma_0(p)$, such that $p$ splits completely in the Hecke field of $f$. Let $\cL,\cL'$ and $\vareps,\vareps'$ be their respective $\cL$-invariants and Atkin--Lehner signs.

\begin{cor}\label{cor:chen}
Assume that:
\begin{itemize}
\item $p\le 7$;
\item Conjecture \ref{locconj} holds for $p$;
\item $v_p(\cL)=v_p(\cL')$.
\end{itemize}
Then $f$ and $g$ are congruent modulo $p^{h+1}$, with $h=v_p(\vareps\cL-\vareps'\cL')-v_p(\cL)$.
\end{cor}



Based on Corollary \ref{cor:chen}, one might wonder if deep congruences are just a small prime phenomenon. However, we detect the same patterns for primes $p\ge 11$, for which Proposition \ref{loctoglob} does not hold in virtue of \cite[Proposition 1.9]{chenquelques} (note however that it can still hold for specific eigenforms, for instance those in \cite[Example 6.3]{anbogrtr}).

\section{Same-sign congruences in small weight}\label{secsamesign}

While the congruences between opposite-sign representations predicted by \Cref{locconj} remain quite obscure, the situation for same-sign congruences is quite simpler. In particular, we show that for small weights, the $\cL$-invariant controls congruences between semistable representations of $G_{\Q_p}$ of the same Atkin--Lehner sign. 
We rely on the fact that such representations live in families (of constant residual representation) locally on a rigid projective line, and on the constancy results modulo powers of $p$ proved for such families in \cite{contor}. The main result of the section is Proposition \ref{samesigncong}. 
It is unclear whether this kind of result is helpful in understanding the opposite-sign congruences predicted by \Cref{locconj}, since we do not know of a way to move between the projective lines parameterizing families of opposite sign that would also allow one to control the mod $p^n$ reduction.

Let $k>2$, so that the representation $V_{k,\cL,\vareps}$ is now irreducible. 
For $n\in\Z$, let $X_n$ be the locus in $\bP^1$ where $v_p(\cL)=-n$: it is an affinoid annulus of inner and outer radius $p^{-n}$, i.e. $\Spm A_n$ for
\[ A_n=\Q_p\langle p^{-n}\cL,p^{-n}\cM\rangle/(\cL\cM-p^{2n}). \]
We denote by $A_n^+$ the ring of power-bounded functions in $A_n$.

We assume from now on that $k$ is even and $2<k<p$. Let $n$ be an integer satisfying $-\frac{k}{2}+2\le n<0$. Then \cite[Théorème 1.2]{bremez} implies the following.

\begin{prop}\label{VLmodp}
The residual representation $\ovl V_{k,\cL,\vareps}$ is multiplicity-free and constant as $\cL$ varies over $X_n(\Qp)$.
\end{prop}

\begin{proof}
In the notation of \cite{bremez}, we are in the case when $\ell(f)$ is a negative integer and $k>2$, i.e. the last case of \cite[Théorème 1.2(iii)]{bremez}. The result follows from the fact that the isomorphism class of $\ovl V_{k,\cL,\vareps}$ given in \emph{loc. cit.} only depends on $a$ and $v_p(\cL)$, and that the two characters appearing on its diagonal are clearly distinct.
\end{proof}

By Proposition \ref{VLmodp} and \cite[Theorem 4.14]{contor}, there exists a continuous representation $\rho_n\colon G_{\Q_p}\to\GL_2(A_n^+)$ that specializes to a lattice in $V_{k,\cL,\vareps}$ at every $\cL\in X_n(\Qp)$. 

\begin{prop}\label{samesigncong}
Let $h\in\Z_{\ge 2}$. If $\cL_0,\cL_1\in X_n(\Q_p)$ satisfy $v_p(\cL_0-\cL_1)>p^{n+h-1}$, then $V_{k,\cL_0,\vareps}$ and $V_{k,\cL_1,\vareps}$ are physically congruent modulo $p^{h}$. 
\end{prop}

Note that the exponent $n+h-1$ is equal to $-v_p(\cL_0)+h-1$.

\begin{proof}
Consider the isomorphism of $\Q_p$-affinoid algebras
\[ g\colon B_n\coloneqq\Q_p\langle\zeta_1,\zeta_2\rangle/(\zeta_1\zeta_2-1)\to A_n=\Q_p\langle p^{-n}\cL,p^{-n}\cM\rangle/(\cL\cM-p^{2n}) \]
mapping $\zeta_1$ to $p^{-n}\cL$ and $\zeta_2$ to $p^{-n}\cM$. The preimage $g^{-1}(A_n^+)$ is contained in the ring $B_n^+$ of power-bounded elements in $B_n$.
We apply \cite[Example 5.18 and Theorem 5.19]{contor} to the representation
\[ g^\ast\rho_n\colon G_{\Q_p}\to\GL_2(B_n^+), \]
taking $m=0$ and $e=1$ in the notation of \emph{loc. cit.} (but note that $n$ is the depth of congruence there, while $n$ is fixed for us and $h$ is the depth of congruence), and obtain that, for every $x_1\in\Spm B_n$ and $h\in\Z$, $h\ge 2$, $g^\ast\rho_n$ is pointwise constant mod $p^h$ on a wide open disc of center $x_1$ and radius $\min\{p^{1-h},p^{1-h}\lvert x_1\rvert\}=p^{1-h}$. We deduce that, for every $\cL_0\in B_n$ and $h\in\Z$, $h\ge 2$, $\rho_n$ is pointwise constant mod $p^h$ on a disc of center $x_1=p^{-n}\cL_0$ and radius $p^{1-h-n}$.
\end{proof}


\begin{cor}\label{loccor}
Let $\vareps\in\{\pm 1\}$, $h\in\Z_{\ge 2}$, and $\cL,\cL^\prime$ such that $n\coloneqq v_p(\cL)\in\Z$ and $v_p(\cL-\cL^\prime)\ge n+h$. Then $V_{k,\cL,\vareps}$ and $V_{k,\cL^\prime,\vareps}$ are physically congruent (hence trace-congruent) modulo $p^h$. 
\end{cor}


\begin{proof}
The result follows by applying Proposition \ref{samesigncong} to $L=\Q_p$, so that $\gamma_{L/\Q_p}(h)=h$, and to $\cL=\cL_0$, $\cL^\prime=\cL_1$.
\end{proof}


\begin{rem}
The results of this section can likely be extended to any $k$ (not necessarily even) in the range $3\le k\le p+1$ via a study of the mod $p$ reductions computed in \cite{chitraoghate}. We didn't do this since it is unclear at the moment whether the result would have interesting consequences, and since the case of $k$ even corresponds to our global setting.
\end{rem}

\begin{rem}
In principle, one could apply \cite[Theorems 4.14 and 5.19]{contor} to any wide open rigid subspace of the projective line on which the residual representation is constant, even if $k$ and $v_p(\cL_0)$ do not satisfy the assumptions of Proposition \ref{samesigncong} (see for instance \cite[Proposition 6.10]{contor}). However, in such a generality we do not know of an analogue of \cite[Théorème 1.2]{bremez} describing the residual representation in terms of the $\cL$-invariant, hence we cannot give an explicit estimate of the domain of $\bP^1$ over which the mod $p^h$ representation is constant. Moreover, while in the Breuil--Mézard setting the domains of constant mod $p$ reduction are as simple as possible (discs or annuli), they can be arbitrarily more complicated in higher weight: 
Rozensztaijn \cite{rozencrys} shows that they are always ``standard subdomains'' of $\bP^1$, i.e. wide open discs minus unions of closed discs, but the number of closed discs appearing in this description can become arbitrarily large as $k$ grows.
\end{rem}

\begin{rem}\label{remord}
If $k=2$, then the representation $V_{2,\cL,\vareps}$ is reducible, an extension of the unramified character $\chi_2=\lambda_\vareps$ mapping a lift of Frobenius to $a$, by $\chi_1=\chi_\cyc\lambda_\vareps$. The isomorphism class of this extension depends on the value of $\cL$, but $\chi_1,\chi_2$ do not, so that the trace of $V_{2,\cL,\vareps}$ is constant as $\cL$ varies over $\bP^1(\Qp)$, and by conjugating with diagonal matrices of entries $p^n, 1$ one can make any two $V_{2,\cL,\vareps}$ and $V_{2,\cL^\prime,\vareps}$ congruent modulo an arbitrarily large power of $p$. Therefore the above argument does not produce any interesting congruence in the case when $k=2$.
\end{rem}

\section{Computational evidence}
In order to obtain computational evidence supporting our global conjecture, there are two main computational challenges to be completed: Both the $\mathcal{L}$-invariants attached to newforms as well as the higher congruences between pairs of newforms need to be computed efficiently.
\subsection{Computing $\mathcal{L}$-invariants}
Thankfully, the efficient computation of $\mathcal{L}$-invariants has been addressed extensively in the literature, see \cite{anbogrtr}, \cite{bplinv} and \cite{graefcon}. Our computations rely on the following two approaches:
\begin{itemize}
\item The algorithm underlying the work of Bergdall-Pollack in \cite{bplinv}, which efficiently computes both classical and $p$-adic $L$-values using modular symbols and then applies the formula from Mazur--Tate--Teitelbaum's exceptional zero conjecture to obtain $\mathcal{L}$-invariants. Bergdall-Pollack kindly shared their highly optimized code with us, allowing the computation of (valuations of) $\mathcal{L}$-invariants in very high weights. 
\item For the direct comparison of $\mathcal{L}$-invariants needed in part $(ii)$ of Conjecture \ref{globconj}, we use the work of the second author together with Samuele Anni, Gebhard Böckle and \'Alvaro Troya in \cite{anbogrtr}, which computes $\mathcal{L}$-invariants using the Greenberg--Stevens formula. 
\end{itemize}
Thus, the main computational challenge in the present work is the efficient computation of higher congruences between modular forms, which we will explain in more detail in the next subsection.

\subsection{Computing deep congruences}\label{comps}
%
%
Let $f\in\mathcal{S}_k(\Gamma_0(N))^\mathrm{new}$ be a newform, where $p$ divides $N$ exactly. As before and throughout this section, we assume that Assumption \ref{main-ass} is satisfied. In particular, this means that the Galois conjugates of $f$ are in one-to-one correspondence with the primes $\mathfrak{p}$ of $K_f$ dividing $p$. Let $\mathcal{O}$ be the ring of integers of $K_f$. Fix a prime $\mathfrak{p}$ of $\cO$ dividing $p$. By assumption, we have $[\mathcal{O}/\mathfrak{p}:\Z/p]=f_\mathfrak{p}=1$. As an immediate consequence, we obtain that the natural embedding (note that $e_\mathfrak{p}=1$)
\[
\Z/p^m\hookrightarrow \mathcal{O}/\mathfrak{p}^m
\]
is an isomorphism, i.e. the $\text{mod } \mathfrak{p}^m$ reduction of $f$ is defined over $\Z/p^m$. By the correspondence between primes above $p$ and Galois conjugates of $f$, the same is true for each conjugate form. Thus, we only need to compute the Hecke eigensystem of $f$, reduce it modulo $\mathfrak{p}^m$ for each prime $\mathfrak{p}$ above $p$ and compare the results in $\Z/p^m$. The only remaining question is how many Hecke eigenvalues we need to compute in practice. For this, we can use the usual Sturm bound after taking into account that we are working away from $Np$. Let
\[
B(k,N)\coloneqq \frac{kI(Np)}{12}-\frac{I(Np)-1}{Np}
\]
with $I(Np)\coloneqq[\SL_2(\Z):\Gamma_0(Np)]=Np\prod_{q\mid Np}\left(1+\frac{1}{q}\right)$, where $q$ varies over the prime divisors of $Np$.

\begin{prop}[Sturm bound]
Let $f,g\in\mathcal{S}_k(\Gamma_0(N))$ be two eigenforms with associated eigensystems $(a_\ell(f))_{\ell\nmid N}$ and $(a_\ell(g))_{\ell\nmid N}$ away from $N$. Assume that their reductions are defined over $\Z/p^m$. If
\[
a_\ell(f)\equiv a_\ell(g) \mod p^m
\]
for all $\ell\nmid Np$, $\ell\leq B(k,N')$ with $N'\coloneqq Np^2\cdot\prod_{q\mid Np}q$. Then 
\[
a_\ell(f)\equiv a_\ell(g) \mod p^m
\]
for all $\ell\nmid Np$.
\end{prop}

\begin{proof}
This follows directly by combining the Sturm bound modulo $p^m$ in \cite[Theorem 3.8]{taiwie} with \cite[Lemma 4.6.5]{miyake}.
\end{proof}

\noindent Most of our computations happen for $N=p$: in this case, $N'=p^4$, $I(N'p)=I(p^5)=p^5+p^4$ and 
\[ B(k,N')=\frac{k(p^5+p^4)}{12}-\frac{p^5+p^4-1}{p^5}<\frac{k(p^5+p^4)}{12}. \]

\begin{rem}\label{rem:main-ass}
We should comment on Assumption \ref{main-ass}. First off, if one is interested in removing this assumption, things get significantly more involved, see for example \cite{taiwie}, which aims at computing reductions in complete generality. The downside of the approach in \cite{taiwie} is that it can only detect if there is a congruence between two members of different newform orbits, but it cannot isolate the actual members within the orbit. Since this is precisely what we are after, we had to use a different method. If $m=1$, one can replace Assumption \ref{main-ass} by just checking via Fermat's little theorem if the reduction is defined over $\Z/p$ and only working with those that are. Unfortunately, there is no analogue for $m>1$. 

In practice, Assumption \ref{main-ass} is at least often satisfied. For example, for $N=p=3$, the first violation of Assumption \ref{main-ass} occurs in weight $46$. 
So even in this very restrictive setup, we are able to collect significant data.
\end{rem}

\subsection{Data}\label{sec:data}

We give some samples of the data we computed. For each $p=3,5,7,11$, we give one table exhibiting deep congruences between $p$-new eigenforms of level $\Gamma_0(p)$ and a sample weight $k$. Our computations go up to weight over 200 for $p=3$, 150 for $p=5$, and 100 for $p=7,11$. Even if \Cref{globconj} is only stated in trivial tame level, we discovered the same patterns in the (smaller) data we computed for higher tame level; we present here some data for $p=3$ in level $\Gamma_0(6)$. More of the data that we computed is available on the repository \cite{deepcong}.

The tables can be read as follows: the value of $N$ corresponds to the level $\Gamma_0(N)$ of the forms involved; the integer in the first column corresponds to the exponent of a power of $p$, while the integers on the second column vary among all of the valuations of the $\cL$-invariants of the $p$-new eigenforms of level $\Gamma_0(p)$ and the given weight $k$. For each integer $n$ in the left entry, every collection of valuations encased by square bracket corresponds to all of the eigenforms whose Hecke eigenvalues belong to a single congruence class modulo $p^n$.

We also indicate the value of the constant $C_{p,k}$ from \eqref{Cpk} for each choice of $p$ and $k$: the forms whose $\cL$-invariants have valuations larger than $C_{p,k}$ do not fit into the pattern.

\begin{rem}\label{pairs}
Even when a certain $\cL$-invariant valuation appears with (even) multiplicity higher than 2, the corresponding eigenforms still split in pairs of congruent ones modulo a sufficiently high-power of $p$: in other words, the valuation of the $\cL$-invariant by itself does not seem to control which eigenforms share a deep congruence.
\end{rem}

\begin{table}[h]
\centering
\caption{$N=p=3, k=44$ $(-C_{p,k}=-8)$}
\begin{tabular}{|l|l|}
\hline
Depth $p^*$ & $v_p(\cL)$ \\ 
\hline
1 & $[0,-6,-6,-8,-8,-11,-11]$ \\  
\hline
2 & $[0, -8, -8], [-6, -6, -11, -11]$ \\
\hline
4 & $[0], [-6, -6], [-8, -8], [-11, -11]$ \\
\hline
9 & $[0], [-6], [-6], [-8, -8], [-11, -11]$ \\
\hline
11 & $[0], [-6], [-6], [-8], [-8], [-11,-11]$\\
\hline
15 & all distinct\\
\hline
\end{tabular}
\end{table}

\begin{rem}
The table for $N=p=3, k=48$ looks exactly the same as the one for $k=44$. This is in agreement with \cite[Conjecture 6.5]{anbogrtr}, and suggests a strengthening of that conjecture so as to include the fact that there is a bijection between newforms in the two weights that not only respects the valuations of the $\cL$-invariants, but also the depth of congruence between any two forms. It might be reflected by an isomorphism between the $\Z_p$-Hecke algebras acting on the spaces of forms in the two weights. 
\end{rem}



\begin{table}[h]
\centering
\caption{$N=p=5, k=32$ $(-C_{p,k}=-6)$}
\begin{tabular}{|l|l|}
\hline
Depth $p^*$ & $v_p(\cL)$ \\
\hline
1 & $[-1,-2,-2,-5,-5,-6,-6,-10,-10,-11,-11]$ \\
\hline
2 & $[-1],[-2,-2],[-5,-5],[-6,-6],[-10,-10],[-11,-11]$\\
\hline
4 & $[-1],[-2],[-2],[-5,-5],[-6,-6],[-10,-10],[-11,-11]$\\
\hline
7 & $[-1],[-2],[-2],[-5],[-5],[-6,-6],[-10,-10],[-11,-11]$\\
\hline
8 & $[-1],[-2],[-2],[-5],[-5],[-6],[-6],[-10,-10],[-11,-11]$\\
\hline
12 & $[-1],[-2],[-2],[-5],[-5],[-6],[-6],[-10],[-10],[-11,-11]$ \\
\hline
14 & all distinct\\
\hline
\end{tabular}
\end{table}




\begin{table}[h]
\centering
\caption{$N=p=7, k=20$ $(-C_{p,k}=-5)$}
\begin{tabular}{|l|l|}
\hline
Depth $p^*$ & $v_p(\cL)$ \\ 
\hline
1 & $[ 0,-1, -5, -5 ], [-1,-6, -6], [ -3, -3 ]$\\
\hline
2 & $[0,-1], [-1], [ -3, -3 ], [-5, -5 ], [-6, -6]$\\
\hline
3 & $[0], [-1], [-1], [ -3, -3 ], [-5, -5 ], [-6, -6]$\\
\hline
5 & $[0], [-1], [-1], [ -3], [-3 ], [-5, -5 ], [-6, -6]$\\
\hline
7 & $[0], [-1], [-1], [ -3], [-3 ], [-5], [-5 ], [-6, -6]$\\
\hline
8 & all distinct\\
\hline
\end{tabular}
\end{table}


\begin{table}[h]\label{11-18}
\centering
\caption{$N=p=11, k=18$ $(-C_{p,k}=-5)$}
\begin{tabular}{|l|l|}
\hline
Depth $p^*$ & $v_p(\cL)$\\
\hline
1 & $[ 0 ], [ 0 ], [ -1, -1 ], [ -3, -3 ], [ -5, -5 ], [ -4,-4,-6, -6], [ -5, -5 ]$\\
\hline
2 & $[ 0 ], [ 0 ], [ -1, -1 ], [ -3, -3 ], [ -4,-4], [ -5, -5 ], [ -5, -5 ], [-6, -6]$\\
\hline
3 & $[ 0 ], [ 0 ], [ -1], [-1 ], [ -3, -3 ], [ -4,-4], [ -5, -5 ], [ -5, -5 ], [-6, -6]$ \\
\hline
5 & $[ 0 ], [ 0 ], [ -1], [-1 ], [ -3], [-3 ], [ -4,-4], [ -5, -5 ], [ -5, -5 ], [-6, -6]$\\
\hline
6 & $[ 0 ], [ 0 ], [ -1], [-1 ], [ -3], [-3 ], [ -4],[-4], [ -5, -5 ], [ -5, -5 ], [-6, -6]$\\
\hline
7 & $[ 0 ], [ 0 ], [ -1], [-1 ], [ -3], [-3 ], [ -4],[-4], [ -5, -5 ], [ -5], [-5 ], [-6, -6]$\\
\hline
8 & all distinct\\
\hline
\end{tabular}
\end{table}


\begin{table}[h]\label{32}
\centering
\caption{$p=3, N=6, k=36$ $(-C_{p,k}=-7)$}
\begin{tabular}{|l|l|}
\hline
Depth $p^*$ & $v_p(\cL)$\\
\hline
1 & $[ -1, -4, -4,-9, -9, -11, -11]$ \\
\hline
2 & $[-1,-9, -9], [-4, -4, -11,-11]$ \\
\hline
5 & $[-1], [-4, -4],[-9, -9], [-11, -11]$ \\
\hline
8 & $[-1], [-4],[-4],[-9, -9],  [-11, -11]$ \\
\hline
13 & $[-1], [-4],[-4],[-9], [-9],  [-11, -11]$ \\
\hline
16 & all distinct \\
\hline
\end{tabular}
\end{table}


\subsubsection{Cancellation between the valuations of the $\cL$-invariants}\label{sec:cancellation}

We observed in Remark \ref{pairs} that the fact that $v_p(\cL(f))=v_p(\cL(g))$ for two $p$-new eigenforms $f,g$ is no guarantee that $f$ and $g$ will share a deep congruence. 
Our Conjecture \ref{locconj} suggests that the difference $v_p(\cL(f)+\cL(g))>v_p(\cL(f))$, i.e. the amount of cancellation happening in the sum of the $\cL$-invariants, give information about the exact power of $p$ appearing in the deep congruence. As one sees from the data below (corresponding to some of the tables presented earlier), the amount of cancellation is always larger than the one prescribed by Conjecture \ref{globconj}. 

Computing $v_p(\cL(f)+\cL(g))$ requires computing both $\cL$-invariants up to some precision, hence is computationally much more taxing than just looking for the individual valuations. For this reason, the data we have available is much scarcer.

The way to read the data is as follows: the first line contains the valuations of the $\cL$-invariants, as given in the previous subsection. The second line shows, for each pair of valuations corresponding to deeply congruent eigenforms, the valuation of $v_p(\cL(f)+\cL(g))$ together with the valuation of the individual $\cL$-invariants. Here is an example with $p=3$:

\medskip

$N=p=3$, $k=44$, $C_{p,k}=-8$

\smallskip

$v_p(\cL): [0, -6, -6, -8, -8, -11, -11]$

\smallskip

$[v_p(\cL(f)+\cL(g)),v_p(\cL(f))]: [-3, -6], [-3, -8], [-3, -11]$ 

\medskip

\noindent Once more, the data for weight 48 is identical. Here is another example, with $p=5$:

\medskip

$N=p=5$, $k=32$, $C_{p,k}=-6$

\smallskip

$v_p(\cL): [-1, -2, -2, -5, -5, -6, -6, -10, -10, -11, -11]$

\smallskip

$[v_p(\cL(f)+\cL(g)),v_p(\cL(f))]: [0, -2], [-1, -5], [-1, -6], [-2, -10], [-2, -11]$




\bigskip

\printbibliography

\bigskip

\begin{small}
\noindent\textsc{Andrea Conti, Heidelberg University,} \url{andrea.conti@iwr.uni-heidelberg.de}

\medskip

\noindent\textsc{Peter Mathias Gr\"af,} 
\url{pgraef03@gmail.com}
\end{small}

\end{document}